\documentclass[11pt]{article}
\usepackage{amsmath,amsthm,amsfonts,amssymb,amscd, amsxtra}
\usepackage{color}
\usepackage{setspace} 
\newtheorem{theorem}{Theorem}
\newtheorem{lemma}[theorem]{Lemma}

\newtheorem{corollary}[theorem]{Corollary}
\newtheorem{proposition}[theorem]{Proposition}

\newtheorem{remark}{Remark}

\begin{document}
\title{An Inexact Newton-like conditional gradient method for constrained nonlinear systems}

\author{
    M.L.N. Gon\c calves
    \thanks{Instituto de Matem\'atica e Estat\'istica, Universidade Federal de Goi\'as, Campus II- Caixa
    Postal 131, CEP 74001-970, Goi\^ania-GO, Brazil.
 (E-mails: {\tt
       maxlng@ufg.br} and {\tt fabriciaro@gmail.com}). The work of these authors was
    supported in part by CAPES,  CNPq Grants  406250/2013-8, 444134/2014-0 and 309370/2014-0.}
  \and F.R. Oliveira  \footnotemark[1]
}
 \maketitle
\begin{abstract}
In this paper, we propose an  inexact Newton-like  conditional gradient method for solving  constrained systems of nonlinear equations.  The  local convergence  of the new method as well as results on its rate are established  by using a general majorant condition.  Two applications of such  condition  are provided: one is for functions whose the derivative satisfies  H\"{o}lder-like  condition and the other is for functions that satisfies a Smale condition, which includes a substantial class of analytic functions. Some preliminaries numerical experiments illustrating the applicability of the proposed method for medium and large  problems  are also presented.
\end{abstract}

\noindent {{\bf Keywords:} constrained nonlinear systems; inexact Newton-like method; conditional gradient method;  local convergence.}

\maketitle
\section{Introduction}

Let $\Omega \subset \mathbb{R}^n$ be an open set,  and $F:\Omega \to \mathbb{R}^n $ be a continuously differentiable nonlinear function. Consider the following constrained system of nonlinear equations
\begin{equation}\label{eq:p}
F(x)=0, \quad x\in C,
\end{equation}
where $C \subset \Omega$ is a nonempty convex compact set. Constrained nonlinear systems such as \eqref{eq:p} appear frequently in many important areas, for instance, engineering, chemistry and economy.
Due to this fact, the numerical solutions of problem \eqref{eq:p} have been the object of intense research in the last years and, consequently, different methods have been proposed in the literature.  Many of them are combinations of Newton methods for solving the unconstrained systems with some strategies  taking into account the constraint set.
Strategies based on  projections, trust region,  active set and gradient methods have  been used; see, e.g., \cite{morini1, bellavia2006, echebest2012, Mangasarian1, Kan2, sandra, cruz2014, mariniquasi, marinez2, Porcelli,wang2016, Zhang1, Zhu2005343}.

Recently, paper \cite{CondG} proposed  and  established a local convergence analysis of a Newton conditional gradient (Newton-CondG) method for solving \eqref{eq:p}. Basically, this method consists  of   computing   the Newton step  and after applying  a conditional gradient (CondG) procedure in order to get the Newton iterative back to the feasible set.  It is important to point out that the CondG method, also known as the Frank-Wolfe method,  is historically known as one of the earliest  first methods  for solving  convex constrained optimization problems,  see~\cite{Dunn1980,NAV:NAV3800030109}.
The  CondG  method and its variants require, at each iteration,  to minimizing a linear  function over the constraint set, which, in general, is significantly simpler than the projection step arising in many proximal-gradient methods. Indeed,  projection problems can be computationally hard for some high-dimensional problems. For instance, in large-scale semidefinite programming, each projection subproblem  of  the proximal-gradient  methods  requires to obtain the complete eigenvalue decomposition of a large matrix while each subproblem of the CondG methods   requires to compute the leading singular vector of such a matrix. The latter requirement  is less computationally expensive (see, for example, \cite{ICML2013_jaggi13} for more details). Moreover, depending on the
application, linear optimization oracles may provide solutions with specic characteristics leading
to important properties such as sparsity and low-rank; see, e.g., \cite{Freund2014,ICML2013_jaggi13} for a discussion on this
subject. Due to these advantages and others, the CondG method  have again received much attention, see for instances
\cite{Freund2014,nemiro2014,ICML2013_jaggi13,lan2015,Ronny2013}.

It  is well-know that  implementations of the Newton method for medium- or large-scale problems may be expensive and difficult due to the necessity to compute all the elements of the Jacobian matrix of $F$, as well as, the exact solution of a linear system for each iteration. For this reason, the main goal of  this work is to present  an  extension of the Newton-CondG method in which  the inexact Newton-like method is considered instead of  standard Newton method. In each step of this new method, the solution of the linear system and Jacobian matrix can be computed in approximate way; see the INL-CondG method in Section~ \ref{sec:condGmet} and comments following it. From the theoretical viewpoint, we present a  local convergence analysis of the proposed method under a majorant condition. The advantage of using  a general condition such as majorant condition in the analyses of Newton methods lies in the fact that it allows to study them    in a unified way. Thus, two applications of majorant  condition  are provided: one is for functions whose the derivative satisfies  H\"{o}lder-like  condition and the other is for functions that satisfies a Smale condition, which includes a substantial class of analytic functions.
From the applicability viewpoint, we report  some preliminaries numerical experiments of the proposed method for medium and large  problems  and compare its performance with the constrained dogleg method \cite{Bellavia2012}.



This paper is organized as follows. Subsection \ref{notation} presents some notation and  basic assumptions.
Section \ref{sec:condGmet} describes  the inexact Newton-like conditional gradient method and presents  its
convergence theorem  whose proof is postponed to Section \ref{sec:PMF}. Two applications of the main convergence theorem are also present in Section \ref{sec:condGmet}. Section \ref{NunEx} presents some preliminary numerical experiments of the proposed method. We conclude the paper with
some remarks.


\vspace{2mm}
\subsection{Notation and basic assumptions}\label{notation}

This subsection presents some notations and assumptions which will be used in the paper. We assume that $F:\Omega \to \mathbb{R}^n $ is a continuously differentiable nonlinear function,
where  $\Omega \subset \mathbb{R}^n$ is an open set containing a nonempty  convex compact set $C$.
The  Jacobian matrix  of $F$ at $x\in \Omega$ is denoted by $F'(x)$.
We also assume that there exists  $x_*\in C$  such that $F(x_*)=0$ and $F'(x_*)$ is nonsingular.
Let  the inner product and its associated norm in $\mathbb{R}^n$ be denoted by $\langle\cdot,\cdot\rangle$ and $\| \cdot \|$, respectively. The open ball centered at $a \in \mathbb{R}^n$ and radius $\delta>0$ is denoted  by $B(a,\delta)$.
For a given linear operator $T:\mathbb{R}^n\to \mathbb{R}^n$, we also use  $\| \cdot \|$ to denote its norm, which is defined by
$
\|T\|:=\sup\{ \|Tx\|,\;\|x\|\leq 1\}.
$
The condition number of a continuous linear operador $A:\mathbb{R}^n\to \mathbb{R}^n$ is denoted by cond$(A)$ and it is defined as
$\mbox{cond(A)}:= \|A^{-1}\| \|A\|$.

\section{The  method and its local convergence analysis}\label{sec:condGmet}
In this section, we  present the inexact Newton-like conditional gradient (INL-CondG) method for  solving~\eqref{eq:p} as well as  its local  convergence theorem whose proof is postponed to Section \ref{sec:PMF}. Our analysis is done by using a majorant condition, which allows to unify the convergence results for two classes of nonlinear functions, namely, one satisfying a H\"{o}lder-like condition and another one satisfying a Smale condition. The convergence results for these special cases are established  in this section.

The INL-CondG method is formally described as follows. \\\\
\fbox{
\begin{minipage}[h]{6.4 in}
{\bf  INL-CondG method}
\begin{description}
\item[ Step 0.] Let $x_0\in C$ and $\{\theta_j\}\subset [0,\infty)$ be given. Set $k=0$ and go to step 1.
\item[ Step 1.] If $F(x_k)=0$, then {\bf stop}; otherwise,  choose an invertible approximation $M_k$ of $F'(x_k)$ and compute  a triple $(s_k,r_k,y_k) \in \mathbb{R}^n\times \mathbb{R}^n\times \mathbb{R}^n$  such that
\begin{equation}\label{aa0}
M_k s_k=-F(x_k) + r_k,  \quad y_{k}=x_k+s_k.
\end{equation}
\item[ Step 2.]  Use CondG procedure  to obtain $x_{k+1}\in C$ as
\begin{equation*}
x_{k+1}=\mbox{CondG}(y_{k},x_k,\theta_k\|s_k\|^2).
\end{equation*}
\item[ Step 3.]  Set $k\gets k+1$, and go to step~1.
\end{description}
\noindent
{\bf end}
\\
\end{minipage}
}
\\

We now describe the subroutine CondG procedure.
\\[2mm]
\fbox{
\begin{minipage}[h]{6.4 in}
{ {\bf CondG procedure} $z=\mbox{CondG}(y,x,\varepsilon)$}
\begin{description}
\item[ P0.]   Set $z_1=x$ and  $t=1$.
\item[ P1.] Use the LO oracle to compute an optimal solution $u_t $ of
$$
g_{t}^*=\min_{u \in  C}\{\langle  z_t-y,u-z_t \rangle\}.
$$
\item[ P2.] If $ g^*_{t}\geq -\varepsilon $,  set $z=z_t$ and {\bf stop} the procedure; otherwise, compute $\alpha_t \in \, (0,1]$ and $z_{t+1}$ as

$$
{\alpha}_t: =\min\left\{1, \frac{-g^*_{t}}{\|u_t-z_t\|^2}  \right\}, \quad  z_{t+1}=z_t+ \alpha_t(u_t-z_t).
$$
\item[ P3.] Set $t\gets t+1$, and go to {\bf P1}.
\end{description}
{\bf end procedure}
\\
\end{minipage}
}
\\
\\[2mm]
{\bf Remarks.} 1) In our local analysis of the INL-CondG method, the invertible approximation $M_k$ of $F'(x_k)$ and the residual $r_k$   will satisfy  classic conditions (see  \eqref{con:qn} and \eqref{con:vk}). The inexact Newton-like method with these conditions on $M_k$ and $r_k$ was proposed  in \cite{morini} and, subsequently,  also studied in, for example,  \cite{chen, MAX1}.
2) The INL-CondG method can be seen as a class of  methods, depending on the choices of  the invertible approximation $M_k$ of $F'(x_k)$  and residual $r_k$. Indeed, by  letting  $M_k =F'(x_k)$ and $r_k=0$, the INL-CondG method corresponds to Newton conditional gradient method which was studied in \cite{CondG}. Another classical choice of $M_k$ would be
$M_k = F'(x_0)$.   We also emphasize that  there are some approach to built $M_k$ that do not involve derivatives, see, for example, \cite{Cruz2006, Luk1998, mariniquasi}.
3) Due to existence of constraint set  $C$,  the point  $y_k$ in Step~1 may be infeasible and hence
the INL-CondG method  use an inexact conditional gradient method in order to obtain the new iterate $x_{k + 1}$ in C.
4) The CondG procedure requires an oracle which is assumed to be able to minimize linear functions over the constraint set.
5) Finally, in the CondG procedure, if $ g^*_{t}\geq -\varepsilon$, then $z_t \in C$ and stop. However, if $ g^*_{t} < -\varepsilon \leq 0$, the procedure continues. In this case, the stepsize $\alpha_t$ is well defined and belong to (0,1].\\

In the following, we state our main local convergence result for the INL-CondG method whose proof is given in Section~\ref{sec:PMF}.
\begin{theorem}\label{th:nt}
Let $x_*\in C$, $R > 0$ and
$
\kappa:=\kappa(\Omega,R)=\sup \left\{ t\in [0, R): B(x_*, t)\subset \Omega \right\}.
$
Suppose that there exist a $f: [0,R) \to \mathbb{R}$ continuously differentiable function such that
\begin{equation}\label{Hyp:MH}
\left\|F'(x_*)^{-1}\left[F'(x)-F'(x_*+\tau(x-x_*))\right]\right\| \leq    f'\left(\|x-x_*\|\right)- f'\left(\tau\|x-x_*\|\right),
\end{equation}
for all $\tau \in [0,1]$ and $x\in B(x_*, \kappa)$, where
\begin{itemize}
  \item[{\bf h1.}]  $f(0)=0$ and $f'(0)=-1$;
  \item[{\bf  h2.}]  $f'$ is strictly increasing.
\end{itemize}
Take the constants $ \vartheta $, $\omega_1$, $\omega_2$ and $\lambda$ such that
\[
0\leq \vartheta < 1,\quad 0\leq \omega_2 < \omega_1, \quad \omega_1 \vartheta + \omega_2 < 1,\quad
\lambda \in \left[0,{(1 - \omega_2 - \omega_1 \vartheta)}/{(\omega_1(1+\vartheta))}\right).
\]
Let the scalars $\nu$,  $\rho$ and $\sigma$ defined as
\begin{equation*}\label{nu:lambda}
\nu:=\sup\{t\in [0, R): f'(t)<0\},
\end{equation*}
\begin{equation*}\label{rho}
\rho:=\sup\left\{\delta \in(0, \nu):\omega_1(1+\vartheta)(1+\lambda)\left(\frac{f(t)}{t f'(t)} -1\right)+\omega_1[(1+\vartheta)\lambda + \vartheta] +\omega_2<1, \; t\in(0, \delta)\right\},
\end{equation*}
\begin{equation}\label{def:sigma}
 \sigma :=\min\{\kappa,\rho\}.
\end{equation}
Let $\{\theta_k\}$ and $x_0$ be given in step 0 of the INL-CondG  method and let also  $\{M_k\}$ and $\{(x_k,r_k)\}$ be generated by  the INL-CondG method. Assume that the invertible approximation $M_k$ of $F'(x_k)$  satisfies
\begin{equation}\label{con:qn}
\|{M_k}^{-1}F'(x_k)\| \leq \omega_{1}, \qquad
\|{M_k}^{-1}F'(x_k)-I\| \leq \omega_{2},  \quad k=0,1,\ldots,
\end{equation}
and the residual  $r_k$ is such that
\begin{equation}\label{con:vk}
\|P_{k}r_{k}\|\leq \eta_{k}\|P_{k}F(x_{k})\|, \qquad
0\leq\eta_{k}\, \mbox{cond}(P_{k}F'(x_{k}))\leq \vartheta, \quad k=0,1,\ldots,
\end{equation}
where $\{P_k\}$ is a sequence of invertible matrix (preconditioners for the linear system in \eqref{aa0}) and $\{\eta_k\}$ is a forcing sequence.
If  $x_0\in C\cap B(x_*, \sigma)\backslash \{x_*\}$ and $\{\theta_k\} \subset [0,\lambda^2/2]$, then $\{x_k\}$  is contained in $B(x_*, \sigma)\cap C$, converges to $x_*$ and there holds
\begin{equation}\label{eq:xklinearconv}
\|x_{k+1}-x_*\| < \|x_{k}-x_*\|, \qquad\limsup_{k \to \infty} \; \frac{\|x_{k+1}-x_*\|}{\|x_k-x_*\|}\leq \omega_1[(1+\vartheta)\sqrt{2\tilde\theta}+\vartheta] + \omega_2,
\end{equation}
where $\tilde \theta=\limsup_{k\to \infty}\theta_k$. Additionally, given $0< p\leq1$, assume that the following condition holds:
\begin{itemize}
\item[{\bf  h3.}] the function  $(0,\, \nu) \ni t \mapsto [ f(t)\big{/} f'(t) - t]/t^{p+1}$ is  strictly increasing.
\end{itemize}
Then, for all integer $k\geq 0$, we have
\begin{multline}\label{eq:xkispconvergent}
\|x_{k+1}-x_*\| \leq
\omega_1(1+\vartheta)(1+\lambda){\left( \frac{f(\|x_0 - x_*\|)}{ f'(\|x_0 - x_*\|)}-\|x_0 - x_*\| \right)}\left(\frac{\|x_k-x_*\|}{{\|x_0 - x_*\|}}\right)^{p+1}\\+(\omega_1[(1+\vartheta)\lambda+\vartheta]+\omega_2)\,\|x_k-x_*\|.
\end{multline}
\end{theorem}

\begin{remark}
As mentioned before, the INL-CondG method as be viewed as a class of methods. Hence, the above Theorem  implies, in particular, the convergence of some new methods, which are named below. We obtain, from Theorem \ref{th:HV}, the convergence for the inexact modified Newton conditional gradient method if $M_k = F'(x_0)$, the inexact Newton conditional gradient method if  $\omega_1 = 1$ and $\omega_2 = 0$  (i.e., $M_k = F'(x_k)$), and the Newton-like conditional gradient method if $\vartheta = 0$ (in this case $\eta_k \equiv 0$ and $r_k \equiv 0$). We also mention that  when $\omega_1 = 1$, $\omega_2 = 0$ and $\vartheta = 0$ (i.e., $M_k =F'(x_k)$, $\eta_k \equiv 0$ and $r_k \equiv 0$),  Theorem \ref{th:HV} is similar to  Theorem~6 in \cite{CondG}.
\end{remark}

\begin{remark}
It is worth pointing out that when $\tilde \theta=0$, then \eqref{eq:xklinearconv} implies that the sequence $\{x_k\}$ converge linearly to $x_*$. Additionally,  if  $\omega_1 = 1$, $\omega_2 = 0$ and $\vartheta = 0$, then $\{x_k\}$ converge superlinear to $x_*$.
On the other hand, if $f'$ is convex, i.e., {\bf h3} holds with $p=1$,   it follows from \eqref{eq:xkispconvergent}, the first inequality in \eqref{eq:xklinearconv}, definition \eqref{def:sigma} and the fact that $x_0\in C\cap B(x_*, \sigma)\backslash \{x_*\}$ that $\{x_k\}$ converge linearly to $x_*$. Additionally, if $\omega_1 = 1$, $\omega_2 = \vartheta = \lambda =0$, it follows from \eqref{eq:xkispconvergent}, that $\{x_k\}$ converge quadratically to $x_*$.
\end{remark}

We now specialize Theorem \ref{th:nt}  for two important classes of functions. In the first one, $F'$ satisfies a H\"{o}lder-like condition \cite{goncalves2013, goncalves2016, huangy2004}, and in the second one,  $F$ is an analytic function  satisfying a Smale condition \cite{Shen201029,S86}.

\begin{corollary}\label{th:HV}
Let $\kappa=\kappa(\Omega,\infty)$ as defined in  Theorem~\ref{th:nt}. Assume that there exist a constant $K>0$ and $ 0< p \leq 1$ such that
\begin{equation}\label{holder_cond}
\left\|F'(x_*)^{-1}[F'(x)-F'(x_*+\tau(x-x_*))]\right\|\leq  K(1-\tau^p) \|x-x_*\|^p, \quad \tau \in [0,1], \quad x\in B(x_*, \kappa).
\end{equation}
Take $0\leq \vartheta < 1$, $0\leq \omega_2 < \omega_1$ such that $\omega_1 \vartheta + \omega_2 < 1$
and $\lambda\in \left[0,{(1 - \omega_2 - \omega_1 \vartheta)}/{(\omega_1(1+\vartheta))}\right)$.
Let $$\bar \sigma:=\min \left\{\kappa, \left[\frac{(1-\omega_1[(1+\vartheta)\lambda +\vartheta]-\omega_2)(p+1)}{K(p-\omega_1[(1+\vartheta)\lambda +\vartheta-p]-\omega_2(p+1)+1)}\right]^{1/p}\right\}.$$
Let $\{\theta_k\}$ and $x_0$ be given in step 0 of the INL-CondG  method and let also $\{M_k\}$ and $\{(x_k, r_k)\}$ be generated by the INL-CondG method. Assume that the invertible approximation $M_k$ of $F'(x_k)$ satisfies
\begin{equation*}
\|{M_k}^{-1}F'(x_k)\| \leq \omega_{1}, \qquad
\|{M_k}^{-1}F'(x_k)-I\| \leq \omega_{2},  \quad k=0,1,\ldots,
\end{equation*}
and the residual  $r_k$ is such that
\begin{equation*}
\|P_{k}r_{k}\|\leq \eta_{k}\|P_{k}F(x_{k})\|, \qquad
0\leq\eta_{k}\, \mbox{cond}(P_{k}F'(x_{k}))\leq \vartheta, \quad k=0,1,\ldots.
\end{equation*}
where $\{P_k\}$ is a sequence of invertible matrix (preconditioners for the linear system in \eqref{aa0}) and $\{\eta_k\}$ is a forcing sequence.
If $x_0\in C\cap B(x_*, \sigma)\backslash \{x_*\}$ and $\{\theta_k\} \subset [0,\lambda^2/2]$, then $\{x_k\}$  is contained in $B(x_*, \sigma)\cap C$, converges to $x_*$ and there hold
\begin{equation*}\label{eq:xk_linearconv}
\|x_{k+1}-x_*\| < \|x_{k}-x_*\|, \qquad \limsup_{k \to \infty} \; \frac{\|x_{k+1}-x_*\|}{\|x_k-x_*\|}\leq \omega_1[(1+\vartheta)\sqrt{2\tilde\theta}+\vartheta] + \omega_2,
\end{equation*}
\begin{equation*}
\|x_{k+1}-x_*\| \leq
\frac{\omega_1(1+\vartheta)(1+\lambda) p K}{(p+1)[1- K\,\|x_0-x_*\|^{p}]}\|x_k-x_*\|^{p+1}+(\omega_1[(1+\vartheta)\lambda+\vartheta]+\omega_2)\|x_k-x_*\|,
\quad   k\geq 0,
\end{equation*}
where $\tilde \theta=\limsup_{k\to \infty}\theta_k$.
\end{corollary}
\begin{proof}
It is immediate to prove that $F$, $x_*$ and $f:[0, \infty)\to \mathbb{R}$ defined by
$
f(t)=Kt^{p+1}/(p+1)-t
$
satisfy the inequality \eqref{Hyp:MH}, conditions {\bf h1}, {\bf h2} and  {\bf h3} in Theorem \ref{th:nt}.  Moreover,  in this case, it is easily seen that $\nu$ and $\rho$, as defined in Theorem \ref{th:nt}, satisfy
$$
\rho=\left[\frac{(1-\omega_1[(1+\vartheta)\lambda +\vartheta]-\omega_2)(p+1)}{K(p-\omega_1[(1+\vartheta)\lambda +\vartheta-p]-\omega_2(p+1)+1)}\right]^{1/p}< \nu = \left[\frac{1}{K} \right]^{1/p},
$$
as a consequence,  $\bar \sigma = \min \{\kappa,\; \rho\}=\sigma$ (see \eqref{def:sigma}). Therefore, the statements of the corollary follow directly from Theorem~\ref{th:nt}.
\end{proof}
\noindent
{\bf Remarks.} 1) If a function $F$ is such that its derivative is L-Lipschitz continuous, i.e., $\|F'(x)-F'(y)\|\leq  L \|x-y\|$, for all $  x, y \in B(x_*, \kappa)$ where $L>0$,
then it also satisfies condition \eqref{holder_cond} with $p=1$ and $K=L\|F'(x_*)^{-1}\|$.
Hence, we obtain the convergence of the INL-CondG method under a Lipschitz condition.
In this case, $\{x_k\}$ converges linearly to $x_*$, and if additionally  $\omega_1 = 1$ and $\omega_2 = \vartheta =\lambda= 0$,
it converges   to $x_*$ quadratically.
2) It is worth mentioning that if $\omega_1 = 1$ and $\omega_2 = \vartheta = 0$ in the previous corollary,  we obtain the convergence of the Newton-CondG method under a H\"{o}lder-like condition, as obtained in \cite[Theorem~7]{CondG}.\\

We next  specialize Theorem~\ref{th:nt}  for the class of analytic functions satisfying a Smale condition.
\begin{corollary}
Let $\kappa=\kappa(\Omega,1/\gamma)$ as defined in Theorem~\ref{th:nt}.
Assume that $F:{\Omega}\to \mathbb{R}^{n}$ is an analytic function and
\begin{equation*} \label{eq:SmaleCond}
\gamma := \sup _{ n > 1 }\left\| \frac
{F'(x_*)^{-1}F^{(n)}(x_*)}{n !}\right\|^{1/(n-1)}<+\infty.
\end{equation*}
Take $0\leq \vartheta < 1$, $0\leq \omega_2 < \omega_1$ such that $\omega_1 \vartheta + \omega_2 < 1$
and $\lambda\in \left[0,{(1 - \omega_2 - \omega_1 \vartheta)}/{(\omega_1(1+\vartheta))}\right)$. Let $a :=\omega_1(1+\vartheta)(1-3\lambda)+4(1-\omega_1\vartheta-\omega_2)$,  $b := 1-\omega_1[(1+\vartheta)\lambda+\vartheta]-\omega_2$ and
$$
\bar \sigma :=\min \left\{\kappa,\frac{a-\sqrt{a^2-8b^2}}{4\gamma b}\right\}.
$$
Let $\{\theta_k\}$ and $x_0$ be given in step 0 of the INL-CondG  method and let also $\{M_k\}$ and $\{(x_k, r_k)\}$ be generated by the INL-CondG method. Assume that the invertible approximation $M_k$ of $F'(x_k)$ satisfies
\begin{equation*}
\|{M_k}^{-1}F'(x_k)\| \leq \omega_{1}, \qquad
\|{M_k}^{-1}F'(x_k)-I\| \leq \omega_{2},  \qquad
\; k=0,1,\ldots,
\end{equation*}
and the residual  $r_k$ is such that
\begin{equation*}
\|P_{k}r_{k}\|\leq \eta_{k}\|P_{k}F(x_{k})\|, \qquad
0\leq\eta_{k}\, \mbox{cond}(P_{k}F'(x_{k}))\leq \vartheta, \qquad k=0,1,\ldots.
\end{equation*}
where $\{P_k\}$ is a sequence of invertible matrix (preconditioners for the linear system in \eqref{aa0}) and $\{\eta_k\}$ is a forcing sequence.
If $x_0\in C\cap B(x_*, \sigma)\backslash \{x_*\}$ and $\{\theta_k\} \subset [0,\lambda^2/2]$, then $\{x_k\}$  is contained in $B(x_*, \sigma)\cap C$, converges to $x_*$ and there holds
\begin{equation*}
\|x_{k+1}-x_*\| < \|x_{k}-x_*\|, \qquad \limsup_{k \to \infty} \; \frac{\|x_{k+1}-x_*\|}{\|x_k-x_*\|}\leq \omega_1[(1+\vartheta)\sqrt{2\tilde\theta}+\vartheta] + \omega_2,
\end{equation*}
\begin{equation*}
\|x_{k+1}-x_*\| \leq
\frac{\omega_1(1+\vartheta)(1 + \lambda)\gamma}{2(1- \gamma \|x_0 - x_*\|)^{2}-1}\|x_k-x_*\|^{2}+(\omega_1[(1+\vartheta)\lambda + \vartheta]+\omega_2){\|x_k-x_*\|},
\quad  k\geq 0,
\end{equation*}
where $\tilde \theta=\limsup_{k\to \infty}\theta_k$.
\end{corollary}
\begin{proof}
Under the assumptions of the corollary, the  real function $f:[0,1/\gamma) \to \mathbb{R}$,
defined by $ f(t)={t}/{(1-\gamma t)}-2t$, is a majorant function for $F$ on $B(x_*, 1/\gamma)$;
see for instance, \cite[Theorem~15]{MAX1}. Since $f'$ is convex, it  satisfies   {\bf h3} in Theorem \ref{th:nt} with $p = 1$; see  \cite[Proposition~7]{MAX1}.
Moreover,  in this case, it is easily seen that $\nu$ and $\rho$, as defined in Theorem~\ref{th:nt},  satisfy
$$\rho=\frac{ a-\sqrt{a^2-8b^2}}{4\gamma b}, \qquad \nu=\frac{\sqrt{2}-1}{\sqrt{2}\gamma}, \qquad \rho<\nu< \frac{1}{\gamma},
$$
and, as a consequence,  $\bar \sigma=\min \{\kappa,\; \rho\}= \sigma$  (see \eqref{def:sigma}). Therefore, the statements of the corollary follow from Theorem~\ref{th:nt}.
\end{proof}
\noindent
{\bf Remark.}  The convergence of the Newton-CondG method under a Smale condition, as obtained  in \cite[Theorem~8]{CondG}, follows from the previous corollary with  $\omega_1 = 1$ and $\omega_2 = \vartheta = 0$.

\section{Proof of Theorem~\ref{th:nt}} \label{sec:PMF}

The main goal of this section is to prove  Theorem 1. First,  we  establish some properties involving the majorant function
and its Newton iteration map. Then, some properties of the CondG procedure are discussed. Finally, the desired proof is presented.


From now on, we assume that all the assumptions of Theorem \ref{th:nt} hold, with the exception of \textbf{h3},
which will be considered to hold only when explicitly stated.


\begin{proposition}  \label{pr:incr1}
The constant $\nu$ is positive and $f'(t)<0 $ for all  $t\in [0,\nu).$ As a consequence, the  Newton iteration map $n_f:[0,\, \nu)\to \mathbb{R}$ defined by
\begin{equation}\label{eq:def.nf}
n_f(t)= t-f(t)/f'(t)
\end{equation}
is well defined and satisfies
\begin{equation}\label{a23}
n_f(t)<0\;   \mbox{ for all }\; t \in (0,\,\nu) \; \mbox{ and } \; \lim_{t\downarrow 0}\frac{|n_f(t)|}t=0.
\end{equation}
Moreover, the constants $\rho$ and $\sigma$ are positive and
\begin{equation}\label{eq:001}
0< \omega_1(1+\vartheta)(1+\lambda)|n_f(t)|+ (\omega_1[(1+\vartheta)\lambda + \vartheta]+\omega_2) t < t,\quad  t\in (0, \, \rho).
\end{equation}
\end{proposition}
\begin{proof}

Firstly, since $f'$ is continuous  and  $f'(0)=-1$, it follows  that the constant  $\nu$ is positive. Hence,   {\bf h2} implies that $f'(t)<0$ for all $t\in [0,\nu)$, from which we conclude  that  $n_f$ is well defined.
On the other hand,  in view of {\bf h2},  we have $f$ is strictly convex in $[0,R)$.
Therefore, using $\nu\leq R$,  we obtain $f(0)>f(t)+f'(t)(0-t),$ for any $t\in  (0,\, \nu)$
which, combined with $f(0)=0$ and  $f'(t)<0$ for any $t\in (0, \nu)$, proves the inequality in \eqref{a23}.
Now, using  the fact that $f(0)=0$ and $n_f(t)<0$ for all $t \in\,  (0,\,\nu)$, we obtain
\begin{equation} \label{eq:rho}
\frac{|n_f(t)|}{t}= \frac{1}{t}\left(\frac{f(t)}{f'(t)}-t\right)=\frac{1}{f'(t)} \left(\frac{f(t)-f(0)}{t-0}\right)-1, \quad  t\in (0,\,\nu).
\end{equation}
Since  $f'(0)\neq 0$, the second statement in \eqref{a23} follows by taking limit in~\eqref{eq:rho}, as $t$  $\downarrow 0$.

It remains to prove the last part of the proposition. First,
as $\lambda<[1 - \omega_2 - \omega_1 \vartheta]/ \omega_1(1+\vartheta)$, we have
$[1-\omega_1(1+\vartheta)\lambda - \omega_1\vartheta - \omega_2] / \omega_1(1+\vartheta)(1+\lambda) > 0$.
Hence, using \eqref{a23}, we conclude that there exists
$\delta>0$ such that
$$
0<\frac{|n_f(t)|}{t}<\frac{1- \omega_1(1+\vartheta)\lambda -\omega_1\vartheta -\omega_2}{\omega_1(1+\vartheta)(1+\lambda)},
 \quad  t\in \, (0, \delta),
$$
or, equivalently,
$$
0<\omega_1(1+\vartheta)(1+\lambda)\frac{|n_f(t)|}{t}+\omega_1[(1+\vartheta) \lambda +\vartheta] +\omega_2  < 1,\quad   t\in\,  (0, \delta).
$$
Hence,  $\rho$  is positive which in turn implies that $\sigma$ is positive and  \eqref{eq:001}  holds.
\end{proof}

The following lemma gives the some relationships between the majorant function $f$ and the nonlinear operator $F$.


\begin{lemma} \label{lemgeral:relateF_nf}
Let  $x \in  B\left(x_*,\min\{\kappa,\nu\}\right)$. Then the function $F'(x)$ is invertible  and the following estimates hold:
\begin{itemize}
\item [a)]  $\|F'(x)^{-1}F'(x_*)\|\leqslant  1/|f'(\| x-x_*\|)|;$
\item [b)] $\|F'(x)^{-1}F(x)\|\leq f(\|x-x_*\|)/ f' (\|x-x_*\|);$
\item [c)] $\|F'(x_*)^{-1}\left[ F(x_*)-F(x)-F'(x)(x_*-x)\right]\|\leq  f' (\|x-x_*\|)\|x-x_*\|-f(\|x-x_*\|)$.
\end{itemize}
\end{lemma}
\begin{proof}
The proof follows the same pattern as the proofs of Lemmas 10, 11 and 12 in~\cite{MAX1}.
\end{proof}
The next result presents a basic property of  the CondG procedure whose the proof can be found  in \cite[lemma~4]{CondG}.
\begin{lemma}  \label{pr:condi}
For any $y,\tilde y\in \mathbb{R}^n$, $x, \tilde x \in C$  and $\mu \geq 0$, we have
$$
\| \mbox{CondG}(y,x,\mu) - \mbox{CondG}(\tilde y,\tilde x,0)\|\leq  \|y-\tilde y\|+ \sqrt{2\mu}.
$$
\end{lemma}
Before presenting  the proof of Theorem~\ref{th:nt}, we  first establish a technical result which will be used to prove that
the sequence $\{x_k\}$ is contained in $B(x_*,\sigma)\cap C$ and the sequence $\{\|x_k-x_*\|\}$ is strictly decreasing.


\begin{lemma} \label{l:wdef}
Assume that $x_k \in C\cap B(x_*, \sigma)\backslash \{x_*\}$. Then, for every $k\geq 0$,
\begin{multline}\label{o698}
\|x_{k+1}-x_*\| \leq \omega_1(1+\vartheta)(1+\sqrt{2\theta_k})|n_f(\|x_k-x_*\|)|+(\omega_1[(1+\vartheta)\sqrt{2\theta_k}+\vartheta]+\omega_2)\|x_k-x_*\|,
\end{multline}
where $n_f$ is as in \eqref{eq:def.nf}.
As a consequence,
\begin{equation}\label{eq:contractionCondG}
\|x_{k+1}-x_*\| < \|x_{k}-x_*\|, \quad  k\geq 0.
\end{equation}
\end{lemma}
\begin{proof} First of all, since $\mbox{CondG}(x,x,0)=x,$ for all $x\in C$, it follows from INL-CondG method that
\begin{align*}
\|x&_{k+1} - x_*\|
=\left\|\mbox{CondG}\left(x_k - M^{-1}_k( F(x_k) - r_k), x_k, \theta_k \|M^{-1}_k( F(x_k) - r_k)\|^2 \right)-\mbox{CondG}(x_*, x_*, 0)\right\|.
\end{align*}
Hence, using  the Lemma \ref{pr:condi} with
\[
y = x_k - M ^{-1}_k (F(x_k) - r_k),  \quad   x=x_k, \quad \mu=\theta_k\|M^{-1}_k( F(x_k) - r_k)\|^2, \quad \tilde y= x_*, \quad  \tilde x=x_*,
\]
we obtain
\[
\|x_{k+1} - x_*\|  \leq\|x_k - M^{-1}_k (F(x_k) - r_k) - x_*\|+ \sqrt{2\theta_k} {\|M^{-1}_k( F(x_k)- r_k)\|}.
\]
Now, simple calculus yields
\begin{multline*}
x_k - M_k^{-1}( F(x_k) -r_k) - x_*  \\ = M^{-1}_k\big(F(x_{*})- F(x_k)-F'(x_k)(x_{*}- x_k)\big)+(M^{-1}_k F'(x_k)- I)(x_{*}- x_k)+ M^{-1}_k r_k.
\end{multline*}
Since, $x_k \in C\cap B(x_*, \sigma)\backslash \{x_*\}$, it follows from Lemma~\ref{lemgeral:relateF_nf} that
$F'(x_k)$ is invertible. Thus, combining the last two inequalities, we obtain
\begin{align*}
\left\|x_{k+1} - x_*\right\| &\leq\|M^{-1}_k F'(x_k)\|\|F'(x_k)^{-1}F'(x_{*})\|\|F'(x_{*})^{-1}[F(x_{*})- F(x_k) - F'(x_k)(x_{*} - x_k)]\|\\
&+\|M^{-1}_k F'(x_k) - I\|\|x_* - x_k\|+\|M^{-1}_k F'(x_k)\|\|F'(x_k)^{-1}P^{-1}_k\|\|P_k r_k\|\\
&+ \sqrt{2 \theta_k} \|M^{-1}_k F'(x_k)\|\|F'(x_k)^{-1} F(x_k)\| + \sqrt{2 \theta_k} \|M^{-1}_k F'(x_k)\|\|F'(x_k)^{-1} P^{-1}_k\| \|P_k r_k\|,
\end{align*}
which, combined with  \eqref{con:qn} and \eqref{con:vk}, yields
\begin{align}
\left\|x_{k+1}- x_* \right\| &\leq \omega_{1} \|F'(x_k)^{-1}F'(x_{*})\|\|F'(x_{*})^{-1}[F(x_{*}) - F(x_k) - F'(x_k)(x_{*}- x_k)]\|\nonumber \\ &+\omega_{2}  \|x_k - x_{*}\|+
\omega_{1} \eta_k  \|F'(x_k)^{-1} P^{-1}_k\|\|P_k F(x_k)\|  \nonumber\\
&+ \omega_1\sqrt{2 \theta_k}  \|F'(x_k)^{-1} F(x_k)\| +  \omega_1\eta _k\sqrt{2 \theta_k}  \|F'(x_k)^{-1} P^{-1}_k\| \|P_k F(x_k)\|.\label{a234}
\end{align}
On the other hand, using  the third inequality in \eqref{con:vk}, we find
\begin{align*}
\omega_{1}\eta_k  \|F'(x_k)^{-1}P^{-1}_k\|\|P_k F(x_k)\| &\leq
\omega_{1}\eta_k  \|(P_k F(x_k))^{-1}\|\|P_k F'(x_k)\|\|F'(x_k)^{-1} F(x_k)\|\\
&\leq \omega_{1}\vartheta  \|F'(x_k)^{-1} F(x_k)\|.
\end{align*}
Hence, it follows from \eqref{a234} that
\begin{align*}
\left\|x_{k+1} - x_* \right\| &\leq\omega_{1} \|F'(x_k)^{-1}F'(x_{*})\|\|F'(x_{*})^{-1}[F(x_{*})- F(x_k)- F'(x_k)(x_{*}- x_k)]\|+\omega_{2} \|x_k -x_{*}\| \\
&+\omega_{1} \vartheta  \|F'(x_k)^{-1} F(x_k)\| + \omega_1\sqrt{2 \theta_k}  \|F'(x_k)^{-1} F(x_k)\| + \omega_1\vartheta\sqrt{2 \theta_k}  \|F'(x_k)^{-1} F(x_k)\|.
\end{align*}
Combining the last inequality with items (a), (b) and (c) of Lemma~\ref{lemgeral:relateF_nf}, we conclude that
\begin{align*}
\left\|x_{k+1} - x_* \right\| &\leq \omega_{1}\left(\frac{f'(||x_k - x_*||) ||x_k - x_*|| - f(||x_k - x_*||)}{|f'(||x_k -x_{*}||)|}\right)+\omega_{2}   \|x_k -x_*\|  \\ & +\left(\omega_{1}\vartheta +
\omega_1 \sqrt{2 \theta_k}  +\omega_1 \vartheta \, \sqrt{2 \theta_k}\right)\frac{f(||x_k - x_*||)}{f'(||x_k - x_*||)} .
\end{align*}
The latter inequality, definition of  $n_f$ in \eqref{eq:def.nf} and the fact that $f'(||x_k - x_*||) < 0$ imply that
\begin{align*}
\|& x_{k+1}- x_* \| \\
&\leq \omega_{1} \left( \frac{f(||x_k - x_*||)}{f'(||x_k - x_{*}||)}-||x_k -x_*|| \right)+\omega_{2}  \|x_k -x_{*}\|+ \omega_1 [(1+\vartheta)\sqrt{2\theta_k} + \vartheta]  \frac{f(||x_k - x_*||)}{f'(||x_k - x_*||)}\\
& =  \omega_1 \, |n_{f}(||x_k - x_*||)| + \omega_2  ||x_k - x_*|| + \omega_1 [(1+\vartheta)\sqrt{2\theta_k} + \vartheta]  (|n_{f}(||x_k - x_*||)| + ||x_k - x_*||).
\end{align*}
Hence, inequality \eqref{o698} now follows by simple calculus.
Since $\sqrt{2\theta_k}\leq \lambda$ and $0<\|x_k - x_*\|<\sigma\leq \rho$, it follows from  \eqref{eq:001} with $t=\|x_k - x_*\|$ that
$$
\omega_1(1+\vartheta)(1+\sqrt{2\theta_k})|n_f(\|x_k - x_*\|)|+(\omega_1[(1+\vartheta)\sqrt{2\theta_k}+\vartheta]+\omega_2) \|x_k-x_*\|<\|x_k -x_*\|,
 $$
which, combined with \eqref{o698}, yields \eqref{eq:contractionCondG}.
\end{proof}

We are now ready to prove  Theorem~\ref{th:nt}.
\begin{proof}[\bf Proof of Theorem~\ref{th:nt}:]
Since $x_0\in C \cap B(x_*, \sigma)\backslash \{x_*\}$, combining the first statement of Lemma~\ref{lemgeral:relateF_nf}, inequality \eqref{eq:contractionCondG} and an induction argument, it is immediate to conclude that the sequence $\{x_k\}$ is contained in $B(x_*,\sigma)\cap C$.

Let us prove that $\{x_k\}$ converges to  $x_*$. Since, $\|x_{k}-x_*\|<\sigma\leq \rho$ for all  $k\geq 0$,
the first inequality in \eqref{eq:xklinearconv} follows trivially from  \eqref{eq:contractionCondG}.
Therefore, $\{\|x_{k}-x_*\|\}$ is a bounded strictly decreasing sequence and hence it converges  to some $\ell_* \in [0,\rho)$. Moreover, taking into account that $n_f(\cdot)$ is continuous in $[0, \rho)$, in particular, from \eqref{o698} we have
$$ \ell_{*} \leq \omega_1(1+\vartheta)(1+\lambda)|n_f(\ell_*)|+(\omega_1[(1+\vartheta)\lambda+\vartheta]+\omega_2)\ell_*,$$
which, combined with  \eqref{eq:001}, implies that $\ell_*=0$ and consequently $x_k \rightarrow x_*$.

Now, from \eqref{o698} we obtain
\[
\frac{\|x_{k+1}-x_*\|}{\|x_{k}-x_*\|}\leq \omega_1(1+\vartheta)(1+\sqrt{2\theta_k})\frac{|n_f(\|x_{k}-x_*\|)|}{\|x_{k}-x_*\|}+ \omega_1[(1+\vartheta)\sqrt{2\theta_k}+\vartheta]+\omega_2, \quad  k\geq 0.
\]
In order to prove the asymptotic rate in~\eqref{eq:xklinearconv}, just take the limit superior in the last inequality as $k \to \infty$ and use $\|x_{k}-x_*\|\to 0$, equality in \eqref{a23} and $\limsup_{k\to \infty}\theta_k=\tilde \theta$.

It remains to show the last part of the theorem. For this purpose, let us assume that {\bf h3} holds.
It follows from \eqref{eq:contractionCondG} and \eqref{o698} respectively that, for all $k\geq 0$,  $\|x_k-x_*\| < \|x_0-x_*\|$ and
\[
\|x_{k+1}-x_*\|\leq \omega_1(1+\vartheta)(1+\sqrt{2\theta_k})\frac{|n_f(\|x_k-x_*\|)|}{\|x_k-x_*\|^{p+1}}\|x_k-x_*\|^{p+1}
+(\omega_1[(1+\vartheta)\sqrt{2\theta_k}+\vartheta]+\omega_2)\|x_k-x_*\|.
\]
Therefore, the inequality \eqref{eq:xkispconvergent} follows due to assumption {\bf h3}, $\sqrt{2\theta_k}\leq \lambda$ and the definition of $n_f$ in ~\eqref{eq:def.nf}.
\end{proof}

\begin{remark}
Similar to the analysis in  \cite{CondG},  we could have defined a scalar sequence $\{t_k\}$, associated to the majorant function $f$, such that
\begin{equation}\label{eq:tk majo}
\|x_k - x_*\| \leq t_k, \quad  k\geq 0.
\end{equation}
Indeed,  if   $\{t_k\}$ is defined  as
\begin{equation*} \label{eq:tknk1}
t_0:=\|x_0-x_*\|, \quad  t_{k+1}:=\omega_1(1+\vartheta)(1+\sqrt{2\theta_k})|n_f(t_k)|+(\omega_1 [(1+\vartheta)\sqrt{2\theta_k}+\vartheta]+\omega_2)t_k, \quad  k\geq 0,
\end{equation*}
  it is possible to prove that \eqref{eq:tk majo} holds, and $\{t_k\}$ is well defined, is strictly decreasing and  converges to $0$. Moreover,
$ \limsup_{k\to \infty}t_{k+1}/t_k= \omega_1[(1 + \vartheta)\sqrt{2\tilde \theta} + \vartheta] +\omega_2,$ where $ \tilde \theta=\limsup_{k\to \infty}{\theta_k}$.
\end{remark}

%

\section{Numerical experiments} \label{NunEx}

In this section, we  present the results of some preliminaries numerical tests which show the computational feasibility of the INL-CondG method.  The experiments were carried out on a set of 15 well-known  box-constrained nonlinear systems (i.e., problem~\eqref{eq:p} with $ C=\{x \in \mathbb{R}^n: l \leq x \leq u\}, $ where $l \in \mathbb{R}^{n}$ and $u \in (\mathbb{R}\cup \{\infty\})^n$) with dimensions between $n=400$ and $n=10000$ (see Table~1).
We made three implementations of our INL-CondG method which differ in the way that the  approximation matrices $M_k$'s are  built. In the first implementation, the  matrices $M_k$'s were approximated by finite difference (FD), whereas in the second and third ones, we used the Broyden-Schubert Update (BSU) \cite{Broyden1971, Schubert1970}  and Bogle-Perkins Update (BPU) \cite{Bogle1990}, respectively.  The three resulting methods described above are denoted by FD-INL-CondG, BSU-INL-CondG and BPU-INL-CondG, respectively.
In the implementations of the BSU-INL-CondG and BPU-INL-CondG methods, the  matrices $M_k$'s were approximated by finite difference when $k = 0$ and mod$(k-1, 5) = 0$.
This strategy to periodically compute   the Jacobian matrix  $F'$ seems to be crucial for the robustness of these derivative-free methods.
For a comparison purpose,  we also run the constrained Dogleg solver (CoDoSol) which is a MATLAB package based on the constrained Dogleg method  \cite{Bellavia2012}, and available on  the web site  {\it http://codosol.de.unifi.it}.


The computational results were obtained using  MATLAB R2016a on a 2.5GHz Intel(R)
i5  with 6GB of RAM and Windows 7 ultimate system. The stopping criterion  $\|F(x_k)\|_{\infty}\leq10^{-6}$  was used, and
a failure  was declared if the number of iterations was greater than $300$ or no progress  was detected.
The starting points were defined as   $x_0(\gamma)= l+ 0.25\gamma (u -l) $ where $\gamma=1,2,3$ for problems having finite lower and upper bounds and  $x_0(\gamma) = 10^{\gamma} (1, \ldots , 1)^{T}$ with $\gamma = 0, 1, 2$, for problems with infinite upper bound. However, since $x_0(0)$ is a solution of Pb13, we used $x_0(3)$ instead. In the  implementations of the  INL-CondG method, the error parameter $\theta_k$ was set equal to $10^{-5}$ for all $k$ and  the
CondG Procedure   stopped when either the required accuracy  was obtained  or the maximum of $300$ iterations were performed.
The parameters of CoDoSol were set as the default choice recommended  by the authors (see \cite[Subsection~4.1]{Bellavia2012}). It worth pointing out that the Jacobian matrices in the latter solver are approximated by  finite difference.

\begin{table}[h]
\centering
\caption{Test problems }
\vspace{0.5cm}
{\small
\begin{tabular}{|c|c|c|c|}
\hline
Problem & Name and souce  & n & Box \\
\hline
Pb 1  & Chandrasekhar's H-equation $c = 0.99$ \cite{Kelley} & 400 & $[0,5]$ \\
Pb 2  & Discrete boundary value function \cite[Problem 28]{more1}   & 500 & [-100,100]\\
Pb 3  & Troesch \cite[Problem 4.21]{Vlcek} & 500  & [-1,1]\\
Pb 4  & Discrete integral \cite[Problem 29]{more1}& 1000 & [-10,10]\\
Pb 5  & Trigexp 1 \cite[Problem 4.4]{Vlcek} & 1000 & [-100,100] \\
Pb 6 & Problem 74 \cite{Vlcek2003}              & 1000  & [0,10]\\
Pb 7 & Problem 77 \cite{Vlcek2003}              & 2000  & [0,10]\\
Pb 8 & Function 15 \cite{william2003}              & 2000  & [-10,0]\\
Pb 9  & Tridiagonal exponential \cite[Problem 4.18]{Vlcek} & 2000 & [$e^{-1}$, $e$]\\
Pb 10 & Trigonometric function \cite[Problem 8]{william2003}  & 2000  & [5,15]\\
Pb 11 & Zero Jacobian function \cite[Problem 19]{william2003}  & 2000  & [0,10]\\
Pb 12  & Trigonometric system \cite[Problem 4.3]{Vlcek} & 5000  & [$\pi$, $2\pi$]\\
Pb 13  & Five diagonal \cite[Problem 4.8]{Vlcek} & 5000  & [1,$\infty$]\\
Pb 14 & Seven diagonal \cite[Problem 4.9]{Vlcek} & 5000  & [0,$\infty$]\\
Pb 15  & Countercurrent reactors \cite[Problem 4.1]{Vlcek} & 10000 & [-1,10]\\
\hline
\end{tabular}
}
\end{table}

Table~2 reports the performance of the FD-INL-CondG, BSU-INL-CondG and BPU-INL-CondG methods, and CoDoSol for solving the 15 problems considered.
In  table~2, ``$\gamma$"  is   the scalar used to compute the starting point $x_0(\gamma)$, ``Iter"  is  the number of iterations of the methods  and  ``$\|F\|_{\infty}$" is the infinity norm of $F$ at the final iterate $x_k$. Finally, the symbol ``$*$" indicates a failure.

\begin{table}[h]
\centering
\caption{\footnotesize Performance of the  FD-INL-CondG, BSU-INL-CondG and BPU-INL-CondG methods and CoDoSol}
\vspace{0.5cm}
{\footnotesize
\begin{tabular}{|cc|cc|cc|cc|cc|}
\hline
  &  &  \multicolumn{2}{|c|}{FD-INL-CondG} & \multicolumn{2}{|c|}{BSU-INL-CondG}  &  \multicolumn{2}{|c|}{BPU-INL-CondG} & \multicolumn{2}{|c|}{CoDoSol} \\
\hline
 Problem& $\gamma$   & Iter & $\|F\|_{\infty}$ & Iter & $\|F\|_{\infty}$  & Iter & $\|F\|_{\infty}$ & Iter & $\|F\|_{\infty}$\\
\hline
Pb 1 &1 &5   &$1.39e$-$11$   & 5 & $9.22e$-$7$      & 5 & $6.99e$-$7$  & 7  & $1.74e$-$11$ \\
     &2 &6   &$4.07e$-$9$    & 7 & $3.58e$-$8$      & 7 & $2.42e$-$8$  & 7  & $8.25e$-$7$ \\
     &3 &5   &$1.71e$-$7$    & 7 & $5.88e$-$11$     & 7 & $1.13e$-$9$  & *  &  \\
\hline
 Pb  2 &1 &9  & $3.30e$-$8$   &12 & $4.78e$-$07$    & 12 & $3.10e$-$8$    & 14 & $3.67e$-$10$ \\
       &2 &1  & $2.59e$-$7$   & 1 & $2.59e$-$7$     & 1  & $2.59e$-$7$    & 2  & $2.19e$-$9$ \\
       &3 &9  & $2.22e$-$7$   &12 & $2.92e$-$7$     & 12 & $2.06e$-$8$    & 14 & $9.18e$-$9$\\
\hline
 Pb  3 &1 &6   & $2.18e$-$7$  & 13 & $7.99e$-$7$       & 8   & $3.60e$-$8$   & 9 & $1.91e$-$9$\\
       &2 &7   & $7.71e$-$8$  & 11 & $2.35e$-$7$       & 9   & $6.65e$-$8$   & 6 & $5.58e$-$9$ \\
       &3 &6   & $2.18e$-$7$  & 9  & $2.02e$-$7$       & 8   & $4.03e$-$8$   & 7 & $1.56e$-$7$  \\
\hline
  Pb  4 &1  &5  & $2.66e$-$10$  & 7 & $9.12e$-$12$     & 7  & $2.62e$-$12$    & 9 & $1.34e$-$10$\\
        &2  &3  & $2.72e$-$11$  & 3 & $2.00e$-$9$      & 3  & $2.00e$-$9$     & 3 & $1.04e$-$7$ \\
        &3  &6  & $4.64e$-$11$  & 7 & $2.23e$-$8$      & *  &                 & 9 & $3.49e$-$8$\\
 \hline
   Pb   5 &1 &20  &$2.60e$-$8$   & 169 & $5.46e$-$9$     & 28  & $2.68e$-$8$  & 21 & $9.01e$-$7$ \\
          &2 &9   &$3.40e$-$11$  & 13 & $1.05e$-$7$      & 12  & $8.00e$-$10$ & 10 & $3.46e$-$9$ \\
          &3 &13  &$5.21e$-$9$   & 19 & $3.64e$-$10$     & 17  & $2.78e$-$8$  & 23 & $2.41e$-$8$\\
\hline
Pb  6 &1  &5    &$5.72e$-$7$   & 7  & $1.67e$-$7$   & 7  & $2.01e$-$8$   & 7 & $3.39e$-$7$ \\
      &2  &13   &$2.72e$-$11$  & 42 & $1.29e$-$8$   & 69 & $2.91e$-$9$   & 9 & $2.32e$-$8$\\
      &3  &9    &$2.22e$-$8$   & 25 & $1.73e$-$7$   & 12 & $2.27e$-$8$   &12 & $3.07e$-$8$\\
      \hline
Pb    7 &1 &6  &$2.19e$-$10$   & 7  & $9.34e$-$9$    & 7  & $9.95e$-$9$  & 9  & $4.22e$-$7$\\
        &2 &8  &$6.44e$-$11$   & 10 & $1.34e$-$9$    & 10 & $3.26e$-$8$  & 12 & $3.45e$-$9$\\
        &3 &9  &$7.83e$-$11$   & 11 & $8.40e$-$8$    & 12 & $1.40e$-$11$ & 13 & $2.73e$-$7$\\
        \hline
Pb  8 &1 &7  &$4.75e$-$10$  & 11 & $6.48e$-$8$   & 9  &  $4.36e$-$8$   &13 & $6.70e$-$11$ \\
      &2 &6  &$3.00e$-$7$   & 10 & $1.36e$-$7$   & 8  &  $2.61e$-$7$   &12 & $1.75e$-$7$\\
      &3 &6  &$2.88e$-$13$  & 8  & $9.36e$-$8$   & 7  &  $4.29e$-$7$   &11 & $1.93e$-$9$ \\
      \hline
Pb   9  &1   &2   &$4.84e$-$14$   & 2 & $4.84e$-$14$   & 2  & $4.84e$-$14$    & 8 & $6.03e$-$14$ \\
        &2   &2   &$1.39e$-$13$   & 2 & $1.39e$-$13$   & 2  & $1.39e$-$13$    & 7 & $6.23e$-$14$\\
        &3   &2   &$2.98e$-$14$   & 2 & $2.98e$-$14$   & 2  & $2.98e$-$14$    & 7 & $3.96e$-$14$\\
        \hline
Pb   10 &1   &7    &$6.88e$-$10$  & 9 & $4.46e$-$10$    & 16  & $4.74e$-$9$   & 10 & $2.12e$-$11$\\
        &2   &3    &$1.45e$-$7$   & 4 & $2.22e$-$10$    & 4   & $1.62e$-$10$  & 4  & $3.36e$-$8$\\
        &3   &10   &$5.65e$-$7$   & 14& $7.83e$-$8$     & 73  & $6.41e$-$9$   & 12 & $1.28e$-$8$ \\
        \hline
Pb   11 &1   &17    &$7.28e$-$7$   & 22 & $9.58e$-$7$    & 27 & $7.51e$-$7$   & 22 & $2.56e$-$7$\\
        &2   &18    &$7.28e$-$7$   & 24 & $6.13e$-$7$    & 28 & $9.21e$-$7$   & 23 & $5.56e$-$7$\\
        &3   &19    &$4.09e$-$7$   & 25 & $5.39e$-$7$    & 30 & $9.31e$-$7$   & 24 & $4.60e$-$7$\\
        \hline
Pb   12 &1   &*   &  & *  &           & *   &   & 18 & $2.60e$-$8$\\
        &2   &*   &  & *  &           & *   &   & 17 & $9.00e$-$8$\\
        &3   &*   &  & *  &           & *   &   & 16 & $4.94e$-$9$ \\
        \hline
Pb   13 &1   &8     &$0.00e$+$0$   & 7 &  $0.00e$+$0$   & 5  & $0.00e$+$0$    & 17& $8.72e$-$10$\\
        &2   &12    &$0.00e$+$0$   & 22&  $0.00e$+$0$   & *  &                & * & \\
        &3   &16    &$0.00e$+$0$   & * &                & 7  & $0.00e$+$0$               & * & \\
        \hline
Pb   14 &0   &4     &$3.64e$-$10$  & 5 & $1.65e$-$7$   & 5  & $2.47e$-$8$    & 4 & $4.13e$-$10$ \\
        &1   &12    &$3.69e$-$7$   & 17& $1.32e$-$9$   & 17 & $2.03e$-$11$   & 17& $1.29e$-$8$\\
        &2   &20    &$1.09e$-$12$  & 28& $6.93e$-$10$  & *  &                & 25& $1.29e$-$11$\\
        \hline
Pb   15 &1   &11    &$6.82e$-$9$   & * &                       & 21 & $3.79e$-$7$    & 17& $2.25e$-$9$ \\
        &2   &12    &$3.06e$-$8$   &19 & $6.05e$-$7$           &  * &                & 19& $2.06e$-$9$\\
        &3   &13    &$5.98e$-$9$   &20 & $3.24e$-$7$           &  * &                & 20& $2.10e$-$9$\\
        \hline
\end{tabular}
}
\end{table}
From Table~2, we see that the FD-INL-CondG, BSU-INL-CondG and BPU-INL-CondG methods and CoDoSol successfully ended 42, 40, 37 and 42 times, respectively,  on a total of 45 runs.
The FD-INL-CondG method is comparable to or even slightly better than  CoDoSol, because it required less iterations in 35  cases in which  both methods successfully ended.
This behavior also has been observed in \cite{CondG} for some small and medium  scale problems.
Regarding the  methods whose the $F'$ is not
evaluated at each iteration, the BSU-INL-CondG method solved 3 problems more than   BPU-INL-CondG method, while the BPU-INL-CondG method required less (resp. more) iterations  than BSU-INL-CondG method in 11 (resp. 7) cases in which  both methods successfully ended. Hence, we may say that latter two methods had similar numerical performance and, for some problems, they are comparable to the methods in which all $M_k$'s are approximated  by finite difference. Finally, based on the previous discussion,
 we  conclude that the INL-CondG method  seems to be a promising tool for solving medium and large box-constrained systems of nonlinear equations.



\section*{Final remarks}

We proposed a method for solving constrained systems of nonlinear equations which is  a combination of inexact Newton-like and conditional gradient methods. Under appropriate hypotheses  and using a majorant condition, it was showed that the sequence generated by new method converge locally. Additionally, we were able to provide convergence results for two important classes of nonlinear functions, namely, one is for functions whose the derivative satisfies  H\"{o}lder-like  condition and the other is for functions that satisfies a Smale condition.
In order to show the practical behavior of the proposed method, we  tested it on medium- and large-scale problems from the literature. The numerical experiments  showed that it works quite well and compares favorably with the constrained dogleg method \cite{Bellavia2012}.

\let\OLDthebibliography\thebibliography
\renewcommand\thebibliography[1]{
  \OLDthebibliography{#1}
  \setlength{\parskip}{0.3pt}
  \setlength{\itemsep}{0pt plus 0.3ex}
}

{




{
\scriptsize 	

}
}

\end{document}